\documentclass[letterpaper, 10pt, conference]{ieeeconf}      

\IEEEoverridecommandlockouts                              
\usepackage{amsmath,graphicx,amsfonts,amssymb,epsfig,mathrsfs}
 \usepackage{amsthm}
\usepackage{subcaption}
\usepackage{color}
\usepackage{multirow}
\usepackage{rotating}
\usepackage{algorithm, algpseudocode,algorithmicx,dsfont}
\usepackage{cite}
\usepackage{url,framed,bm}

\newtheorem{theorem}{Theorem}

\newtheorem{remark}{Remark}

\newcommand{\bQ}{\bm{Q}}

\newcommand{\bH}{\bm{H}}
\newcommand{\bI}{\bm{I}}
\newcommand{\bpi}{\bm{\pi}}
\newcommand{\balpha}{\bm{\alpha}}
\newcommand{\bbeta}{\bm{\beta}}
\newcommand{\bgamma}{\bm{\gamma}}
\newcommand{\bc}{\bm{c}}
\newcommand{\bp}{\bm{p}}
\newcommand{\bq}{\bm{q}}
\newcommand{\bx}{\bm{x}}
\newcommand{\differential}{\mathrm{d}}
\newcommand{\prox}{{\rm{prox}}}

\allowdisplaybreaks

\IEEEoverridecommandlockouts
\begin{document}

\title{Proximal Recursion for the Wonham Filter\thanks{Supported in part by NSF under grants 1665031, 1807664 and 1839441, and by AFOSR under grant FA9550-17-1-0435.
}}
\author{Abhishek Halder$^{1}$\thanks{$^{1}$Department of Applied Mathematics, University of California, Santa Cruz,
        CA 95064, USA,
        {\tt\small{ahalder@ucsc.edu}}}%
        , and Tryphon T. Georgiou$^{2}$\thanks{$^{2}$Department of Mechanical and Aerospace Engineering, University of California, Irvine,
        CA 92617, USA,
        {\tt\small{tryphon@uci.edu}}}%
        }
\maketitle

\definecolor{gray}{RGB}{200,200,200}
\newcommand{\gray}{\color{gray}}
\newcommand{\blue}{\color{blue}}

\begin{abstract}
This paper contributes to the emerging viewpoint that governing equations for dynamic state estimation, conditioned on the history of noisy measurements, can be viewed as gradient flow on the manifold of joint probability density functions with respect to suitable metrics. Herein, we focus on the Wonham filter where the prior dynamics is given by a continuous time Markov chain on a finite state space; the measurement model includes noisy observation of the (possibly nonlinear function of) state. We establish that the posterior flow given by the Wonham filter can be viewed as the small time-step limit of proximal recursions of certain functionals on the probability simplex. The results of this paper extend our earlier work where similar proximal recursions were derived for the Kalman-Bucy filter.
\end{abstract}


\section{Introduction}
We consider the problem of estimating the state of a continuous time Markov chain $X(t)$ on finite state space $\Omega = \{a_{1}, \hdots, a_{m}\}$ with $m\times m$ transition\footnote{We suppose that the Markov chain is time-homogeneous, i.e., the associated transition probability matrix is $\exp(\bQ t)$ for all $t\geq 0$.} rate matrix $\bQ$, with
entries $Q_{ij} \geq 0$, for $i\neq j$, and $Q_{ii} = - \sum_{j\neq i} Q_{ij} < 0$.
To ease notation, we hereafter write $X(t) \sim \text{Markov}\left(\bQ\right)$. Suppose that one observes the process $Z(t)$ governed by the It\^{o} stochastic differential equation (SDE)
\begin{align}
\differential Z(t) = h\left(X(t)\right)\:\differential t + \sigma_{V}(t) \:\differential V(t),
\label{Observn}	
\end{align}
where $h(\cdot)$ is a deterministic injective function of the state, $\sigma_{V}(t)$ is continuously differentiable and bounded away from zero for all $t\geq 0$, and the standard Wiener process $V(t)$ is independent of the process $X(t)$. One typically refers to $h(\cdot)$ as the sensing or measurement model, and $V(t)$ as the measurement noise. Given the history of noisy observation $\{Z(s), 0\leq s \leq t\}$, the objective of the estimation problem is to compute the \emph{conditional probability} of the state $X(t)$, i.e., to compute the \emph{posterior probabilities}
\begin{align}
\!\!\!\pi_{i}^{+}(t) \!:=\! \mathbb{P}\big\{X(t)=a_{i} \!\mid\! Z(s), 0\leq s \leq t \big\}, \, i=1,\hdots,m. 
\label{DefPosterior}	
\end{align}

Let the initial occupation probability (row) vector be $\bpi_{0}$ satisfying $\bpi_{0}\geq 0$ element-wise, and $\bpi_{0} {\mathds{1}} = 1$, where ${\mathds{1}}$ denotes a column vector of ones. The time evolution of the \emph{prior distribution} $\bpi^{-}(t) := \{\pi_{1}^{-}(t), \hdots, \pi_{m}^{-}(t)\}$ is governed by the ordinary differential equation (ODE) 
\begin{eqnarray}
\dot{\bpi}^{-}(t) = \bpi^{-}(t) \bQ, \qquad \bpi^{-}(0) = \bpi_{0}. 
\label{PriorEvolution}	
\end{eqnarray}
In other words, (\ref{PriorEvolution}) gives the \emph{unconditional probabilities} of the state $X(t)$, i.e., $\pi_{i}^{-}(t)=\mathbb{P}(X(t)=a_{i})$, $i=1,\hdots,m$.

In \cite{wonham1964some}, Wonham showed that for the state-observation model given by
\begin{subequations}\label{StateObsv}	
\begin{align}
X(t) &\sim \text{Markov}\left(\bQ\right), \label{ProcessModel}\\
\differential Z(t) &= h\left(X(t)\right)\:\differential t + \sigma_{V}(t) \:\differential V(t), \label{ObsModel}	
\end{align}
the posterior probability $\bpi^{+}(t):=\{\pi_{1}^{+}(t), \hdots, \pi_{m}^{+}(t)\}$ evolves according to the It\^{o} SDE
\end{subequations}
\begin{align}
\differential\bpi^{+}(t) = \bpi^{+}(t) \bQ \: \differential t \:+\:&\frac{1}{\left(\sigma_{V}(t)\right)^{2}}\bpi^{+}(t)\left(\bH - \widehat{h}(t)\bI\right)\times\nonumber\\
&\left(\differential Z(t) - \widehat{h}(t)\differential t\right),
\label{WonhamPosteriorSDE}	
\end{align}
with initial condition $\bpi^{+}(0) = \bpi_{0}$, where
\begin{align*}
\bH := {\rm{diag}}\left(h(a_{1}), \hdots, h(a_{m})\right), \quad \widehat{h}(t) := \displaystyle\sum_{i=1}^{m} h(a_{i}) \pi_{i}^{+}(t).
\end{align*}
The vector SDE (\ref{WonhamPosteriorSDE}) has since been known as the Wonham filtering equation that allows computing the conditional probabilities of the state. Reference \cite[eqn. (21)]{wonham1964some} derived (\ref{WonhamPosteriorSDE}) with $h(\cdot)$ as the identity map; the form (\ref{WonhamPosteriorSDE}) has appeared in the literature since then -- for recent references see e.g., \cite[eqn. (2)]{zhang2007two} and \cite[eqn. (5)]{yang2016feedback}. Having obtained $\bpi^{+}$ from (\ref{WonhamPosteriorSDE}), assuming that the points in $\Omega$ are elements of a linear space, one can compute the optimal (in the minimum mean squared error sense) state estimate  given by the conditional expectation
\begin{align}
\widehat{X}(t) := \mathbb{E}\left[X(t) \mid Z(s), 0\leq s \leq t\right] = \sum_{i=1}^{m}a_{i}\pi_{i}^{+}(t).
\label{CondExpectation}	
\end{align}

The purpose of this paper is to give new variational interpretation of the flow $\bpi^{+}(t)$ governed by (\ref{WonhamPosteriorSDE}). Specifically, we seek a gradient flow description for the evolution of the posterior or conditional probability on the standard simplex 
\begin{align}
\Delta^{m-1}:=\{\bpi\in\mathbb{R}_{\geq 0}^{m} \mid \bpi{{\mathds{1}}}=1\}.
\label{Simplex}	
\end{align}
Such interpretations were uncovered in \cite{laugesen2015poisson} for nonlinear filtering with zero prior dynamics and, more generally, in our recent works \cite{halder2017gradient,halder2018gradient} for the Kalman-Bucy filter. Results of such flavor are not only fundamental in systems-theoretic context, but may also be transformative in computation since they open up the possibility to solve the filtering equations via proximal algorithms \cite{rockafellar1976monotone,parikh2014proximal}. This is pursued in \cite{caluya2019proximal,caluya2019gradient} for fast computation of the prior joint probability density functions without spatial discretization.

This paper is structured as follows. In Section \ref{SecMainIdea}, we outline the main ideas for gradient flow formulation via proximal recursion. The recursions for computing the posterior in the Wonham filter are derived in Section \ref{SecProxPosterior}. In Section \ref{SecProxPrior}, proximal recursion for computing the prior is given for the case of reversible Markov chain. Numerical examples are given in Section \ref{SecNumericalExample} for both reversible and non-reversible Markov chains to illustrate the scope of the proposed framework. Section \ref{SecConclusion} concludes the paper.

\subsubsection*{Notations} We use $\circ$ to denote function composition, and $\langle\cdot,\cdot\rangle$ to denote the standard Euclidean inner product. The notation $\nabla_{\bx}$ stands for standard Euclidean gradient w.r.t. vector $\bx$. Furthermore, $\odot$ denotes elementwise multiplication. The notation $\exp(\cdot)$ with vector argument means elementwise exponential, and the same with matrix argument denotes the exponential matrix.


\section{Main Idea}\label{SecMainIdea}
We adopt a metric viewpoint of gradient flow that approximates the flow of probability distribution $\bpi(t)$ starting from a given initial condition $\bpi_{0}:=\bpi(0)$, as the small time-step limit of a variational recursion in the form
\begin{align}
\bp_{k}(\lambda) = \underset{\bp}{\arg\inf}\:\frac{1}{2}d^{2}\left(\bp,\bp_{k-1}\right) \:+\: \lambda\Phi(\bp),
\label{JKOgeneral}	
\end{align}
where $\bp_{0}\equiv\bpi_{0}$, $k\in\mathbb{N}$, and $\lambda$ is the step-size. Here, $d(\cdot,\cdot)$ is a distance functional between two probability distributions, and the functional $\Phi(\cdot)$ depends on the generator of the flow $\bpi(t)$. In particular, the functionals $d(\cdot,\cdot)$ and $\Phi(\cdot)$ are to be chosen such that $\bp_{k}(h)\rightarrow\bpi(t=k\lambda)$ as $\lambda\downarrow 0$.

\begin{figure}[t]
\centering
\includegraphics[width=.45\textwidth]{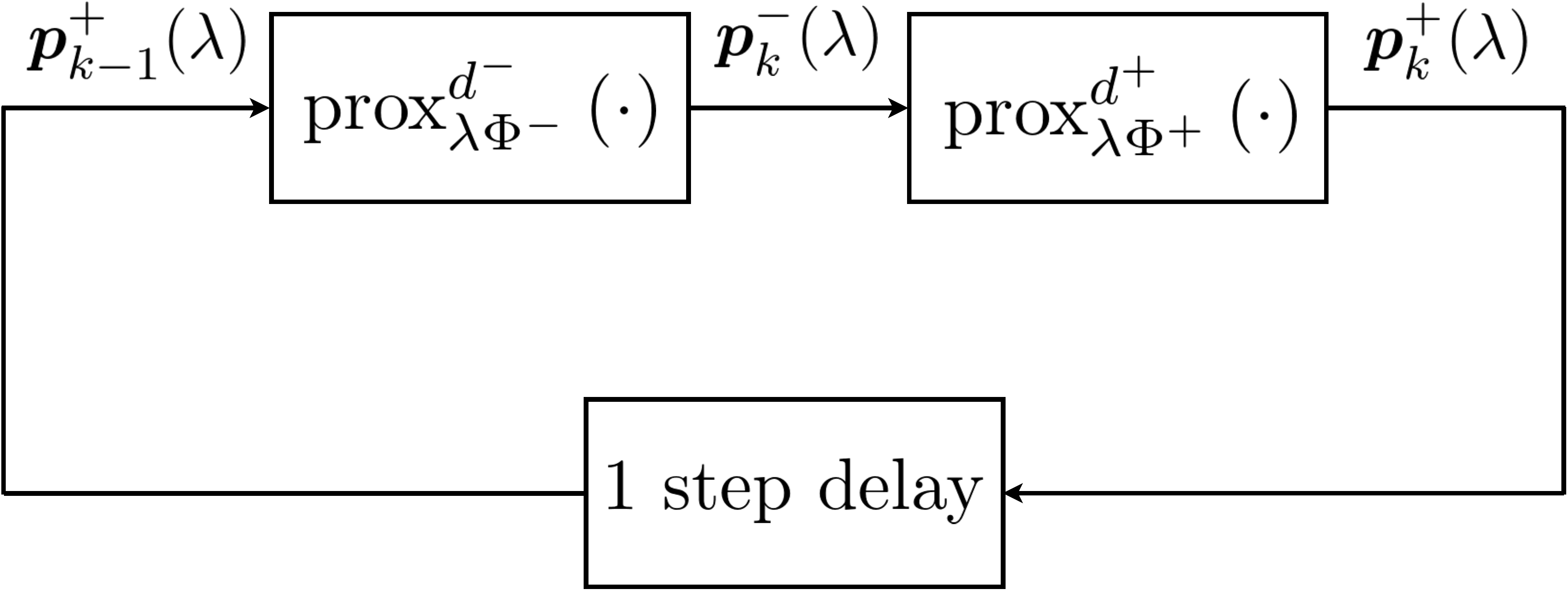}
\caption{{\small{The block diagram showing the prior and posterior computation as sequential proximal updates given by (\ref{PriorPosteriorProx}) for $k\in\mathbb{N}$.}}}
\vspace*{-0.2in}
\label{BlockDiagm}
\end{figure}

The recursion (\ref{JKOgeneral}) is reminiscent of the Euclidean setting, where the gradient flow for the ODE $\dot{\bx} = -\nabla_{\bx}\Phi(\bx)$ can be approximated via a recursion of the form (\ref{JKOgeneral}) with $d(\cdot,\cdot)$ as the Euclidean distance metric, and $\Phi(\cdot,\cdot)$ as in the argument generating the vector field. In the optimization literature, the operators associated with such recursions are termed as \emph{Moreau-Yosida proximal operators} \cite{moreau1965proximite,yosid1964,rockafellar1976monotone,parikh2014proximal}, denoted as
\[\prox_{\lambda\Phi}^{d},
\]
which reads, proximal operator for functional $\lambda\Phi$ w.r.t.\ $d$. Likewise, we use the same notation for the right-hand-side of (\ref{JKOgeneral}) in our more general setting. This allows interpreting the discrete time-stepping as steepest descent of the functional $\Phi$ w.r.t. distance $d$. Proximal operators have also been used in general Hilbert spaces \cite{bauschke2011monotone}, and in the space of probability density functions \cite{halder2017gradient,halder2018gradient,caluya2019proximal,caluya2019gradient,jordan1998variational,ambrosio2008gradient,zhang2018mean}. The idea of applying proximal recursion in the space of probability measures appeared first in \cite{jordan1998variational}; see also \cite{ambrosio2008gradient}.

In the filtering context, we think of the computation of prior followed by that of posterior, as composition of respective proximal operators. Denoting the approximate prior and posterior probability vectors for the proximal recursions associated with Wonham filter as $\bp_{k}^{-}$ and $\bp_{k}^{+}$, respectively, we write
\begin{subequations}\label{PriorPosteriorProx}
\begin{align}
\bp_{k}^{-}(\lambda) &= \prox_{\lambda\Phi^{-}}^{d^{-}}\left(\bp_{k-1}^{+}\right) \nonumber\\
&= \underset{\bp\in\Delta^{m-1}}{\arg\inf}\:\frac{1}{2}\left(d^{-}\left(\bp,\bp_{k-1}^{+}\right)\right)^{2} + \lambda\Phi^{-}(\bp), \label{PriorProx}\\	
\bp_{k}^{+}(\lambda) &= \prox_{\lambda\Phi^{+}}^{d^{+}}\left(\bp_{k}^{-}\right) \nonumber\\
&= \underset{\bp\in\Delta^{m-1}}{\arg\inf}\:\frac{1}{2}\left(d^{+}\left(\bp,\bp_{k}^{-}\right)\right)^{2} + \lambda\Phi^{+}(\bp), \label{PosteriorProx}	
\end{align}
\end{subequations}
where $k\in\mathbb{N}$, $\lambda>0$ is the step-size, and $\left(d^{\pm},\Phi^{\pm}\right)$ are to be determined functional pairs guaranteeing $\bp_{k}^{+}(\lambda) \rightarrow \bpi^{+}(t=k\lambda)$ as $\lambda \downarrow 0$, wherein $\bpi^{+}(t)$ solves (\ref{WonhamPosteriorSDE}). In other words, $\left(d^{\pm},\Phi^{\pm}\right)$ are to be designed such that the composite map $\prox_{\lambda\Phi^{+}}^{d^{+}}\circ\prox_{\lambda\Phi^{-}}^{d^{-}}$ approximates the flow of (\ref{WonhamPosteriorSDE}) in small time-step limit (see Fig. \ref{BlockDiagm}). 

Next, we focus on the problem of designing the pair $\left(d^{+},\Phi^{+}\right)$ in (\ref{PosteriorProx}). 


\section{Proximal Recursion for the Posterior}\label{SecProxPosterior}
We first derive a proximal recursion of the form (\ref{PosteriorProx}) for the posterior update in the special case $\bQ\equiv\bm{0}$ (Section \ref{SubsecZPD}). The proof for the same recovers the explicit stochastic integral formula given in \cite[eqn. (5)]{wonham1964some}. We then show that the same proximal recursion applies for the general $\bQ\neq\bm{0}$ case (Section \ref{SubsecGeneral}).

\subsection{The Case of Zero Prior Dynamics}\label{SubsecZPD}
As in \cite[Section 2]{wonham1964some}, we start with the simple case when the state $X$, instead of being a Markov chain, is a random variable taking values in $\Omega = \{a_{1}, \hdots, a_{m}\}$ with a (known) prior probability distribution $\bp_{0}\in\Delta^{m-1}$ at $t=0$.

For $k\in\mathbb{N}$, let $t_{k-1}:=(k-1)\lambda$ where $\lambda$ is the step size, and let $\{Z_{k-1}\}_{k\in\mathbb{N}}$ be the sequence of samples of the process $Z(t)$ at $\{t_{k-1}\}_{k\in\mathbb{N}}$. Introducing $Y_{k-1} := (Z_{k} - Z_{k-1})/\lambda$, we consider the functional
\begin{align}
\Phi^{+}(\bp) := \frac{1}{2\left(\sigma_{V}(t_{k-1})\right)^{2}}\mathbb{E}\left[\left(Y_{k-1} - h\right)^{2}\right],
\label{PhiPlus}	
\end{align}
where the expectation operator $\mathbb{E}\left[\cdot\right]$ is taken w.r.t. the probability vector $\bp\in\Delta^{m-1}$. The following result shows that with $\Phi^{+}$ as in (\ref{PhiPlus}), the functional $\frac{1}{2}(d^{+})^{2}$ in (\ref{PosteriorProx}) can be taken as the Kullback-Leibler divergence $D_{{\rm{KL}}}$, given by
\begin{align}
D_{{\rm{KL}}}(\balpha\parallel\bbeta) \!:= \!\sum_{i=1}^{m}\!\alpha_{i}\log(\alpha_{i}/\beta_{i}), \;\text{for}\;\balpha,\bbeta\in\Delta^{m-1}.
\label{KLdef}	
\end{align}
In other words, (\ref{PosteriorProx}) can be viewed as an entropic proximal mapping \cite{teboulle1992entropic,censor1992proximal}.
\begin{theorem}\label{ThmPosteriorSplProxWonham}
Let $\Phi^{+}(\bp)$ be as in (\ref{PhiPlus}), and consider the proximal recursion
\begin{align}
\bp_{k}^{+}(\lambda) = \underset{\bp\in\Delta^{m-1}}{\arg\inf}\:D_{{\rm{KL}}}\left(\bp\parallel\bp_{k}^{-}\right) + \Phi^{+}(\bp), \, k\in\mathbb{N},
\label{PosteriorUpdateProx}	
\end{align}
with initial condition $\bp_{0}\in\Delta^{m-1}$. Let $\bpi^{+}(t)$ be the flow generated by (\ref{WonhamPosteriorSDE}) with $\bQ\equiv\bm{0}$ and initial condition $\bpi^{+}(0)\equiv\bp_{0}$. Then $\bp_{k}^{+}(\lambda) \rightarrow \bpi^{+}(t=k\lambda)$ as $\lambda\downarrow 0$.
\end{theorem}
\begin{proof}
See Appendix A.	
\end{proof}


\subsection{The General Case}\label{SubsecGeneral}
In the following, we formally state and prove that in the general case $\bQ\neq\bm{0}$, the recursion (\ref{PosteriorUpdateProx}) still applies for the posterior computation. Compared to the proof of Theorem \ref{ThmPosteriorSplProxWonham}, the proof now will differ since the map $\bp_{k-1}^{+}\mapsto\bp_{k}^{-}$	is no longer identity, and one does not have an analytical solution for the SDE (\ref{WonhamPosteriorSDE}) for $\bQ\neq\bm{0}$, in general. 

\begin{theorem}\label{ThmPosteriorGeneralProxWonham}
Let $\Phi^{+}(\bp)$ be as in (\ref{PhiPlus}), and consider the proximal recursion
\begin{align}
\bp_{k}^{+}(\lambda) = \underset{\bp\in\Delta^{m-1}}{\arg\inf}\:D_{{\rm{KL}}}\left(\bp\parallel\bp_{k}^{-}\right) + \Phi^{+}(\bp), \, k\in\mathbb{N},
\label{PosteriorUpdateProxAgain}	
\end{align}
with initial condition $\bp_{0}\in\Delta^{m-1}$. Let $\bpi^{+}(t)$ be the flow generated by (\ref{WonhamPosteriorSDE}) with initial condition $\bpi^{+}(0)\equiv\bp_{0}$. Then $\bp_{k}^{+}(\lambda) \rightarrow \bpi^{+}(t=k\lambda)$ as $\lambda\downarrow 0$.
\end{theorem}
\begin{proof}
See Appendix B.	
\end{proof}


\section{Proximal Recursion for the Prior}\label{SecProxPrior}
We now derive a proximal recursion of the form (\ref{PriorProx}) for the prior update. We assume that the Markov chain $X(t)$ is irreducible. Then, $X(t)$ has a unique stationary distribution vector $\bpi_\infty$ which is the limit $\lim_{t\to \infty}\bpi^-(t)$ and is elementwise positive.
We further assume that the Markov chain is
reversible, i.e., that the detailed balance condition
\begin{equation}\label{eq:detail}
(\bpi_\infty)_i \bm{Q}(i,j) = (\bpi_\infty)_j \bm{Q}(j,i)
\end{equation}
holds. Here $(\cdot)_i$ denotes the $i$th entry of.

Starting from a known probability distribution 
\[
\bpi^-(t_{k-1})=\bp_{k-1}^+(\lambda), \mbox{ for }t_{k-1}=(k-1)\lambda,
\]
the evolution of the prior $\bpi^-(t)$ of $X(t)$ is governed
by the prior dynamics \eqref{PriorEvolution}.
Thus,
\[
\bpi^- (\underbrace{t_{k-1}+\lambda}_{t_k =k\lambda})=\bp_{k-1}^+(\lambda) \exp\left(\lambda \bm{Q}\right),
\]
equivalently, $\bp_{k-1}^+(\lambda)=\bpi^-(t_k)\exp\left(-\lambda \bm{Q}\right)$.
Hence,
\begin{equation}\label{eq:bpkmlambda}
\bp_{k-1}^+(\lambda)=\bpi^-(t_k)(\bm{I}-\lambda \bQ) + o(\lambda),
\end{equation}
where $o(\lambda)$ signifies the ``order of''.

As a consequence of \eqref{eq:detail}, the transition rate matrix $\bm{Q}$ defines a symmetric operator when considered with respect to the inner product
\begin{align}
\langle \bp,\bq\rangle_{\bpi_\infty}:= \sum_i\frac{ (\bp)_i (\bq)_i}{(\bpi_\infty)_i}. 
\label{DefInnerProduct}
\end{align}
Indeed, if $\bm{D}_{\bpi_\infty}$ denotes the diagonal matrix formed with the entries of $\bpi_\infty$, then in matrix notation, \eqref{eq:detail} becomes
\begin{equation}\label{eq:T}
\bm{D}_{\bpi_\infty}\bQ=\bm{Q}^T \bm{D}_{\bpi_\infty},
\end{equation}
and therefore
$
\langle \bp \bQ,\bq\rangle_{\bpi_\infty}
=\langle \bp,\bq \bQ\rangle_{\bpi_\infty},$
where $^T$ denoting matrix/vector transposition. 
We can now express $\bp_k^-(\lambda)$ as a solution to a proximal
recursion. Define the quadratic form
\begin{equation}
\Phi^-(\bp)= - \frac12\langle \bp \bQ,\bp\rangle_{\bpi_\infty},
\label{Phiminus}
\end{equation}
and denote
\begin{align}
\|\bp\|_{\bpi_\infty}^2=\langle \bp,\bp\rangle_{\bpi_\infty}.
\label{DefNorm}
\end{align}

\begin{theorem} \label{ThmPrior}
Let $\Phi^{-}(\bp)$ be as in (\ref{Phiminus}). The $\lambda$-approximate prior satisfies the following proximal recursion
\begin{equation}\label{eq:quadratic}
\bp_k^-(\lambda) = \underset{\bp\in\Delta^{m-1}}{\arg\inf}\:  \frac12\|\bp-\bp_{k-1}^+\|_{\bpi_\infty}^2+\lambda\Phi^-(\bp).
\end{equation}
\end{theorem}

\begin{proof}
The stationarity condition for \eqref{eq:quadratic} becomes
\begin{align}
\bm{0}&=\frac{\partial}{\partial \bp}\left(\frac12\|\bp-\bp_{k-1}^+\|_{\bpi_\infty}^2+\lambda\Phi^-(\bp)\right)\nonumber\\
& = (\bp-\bp_{k-1}^+)\bm{D}_{\bpi_\infty}^{-1} -\lambda \bp \frac12 \left(\bQ\bm{D}_{\bpi_\infty}^{-1} + \bm{D}_{\bpi_\infty}^{-1}\bQ^T\right)\nonumber\\
& = (\bp-\bp_{k-1}^+)\bm{D}_{\bpi_\infty}^{-1} -\lambda \bp \bQ\bm{D}_{\bpi_\infty}^{-1},\label{laststepPriorProof}
\end{align}
since $\bm{D}_{\bpi_\infty}^{-1}\bQ^T= \bQ\bm{D}_{\bpi_\infty}^{-1}$ from \eqref{eq:T}. Here, by $\bm{0}$ we denote the zero vector of compatible dimensions.
Thus,
\[
\bp_{k-1}^+(\lambda)=\bp(\bm{I}-\lambda\bQ).
\]
For sufficiently small $\lambda$, the matrix $\bm{I}-\lambda \bQ$ is invertible and the unique minimizer $\bp$ has positive entries.
Moreover, since $\bQ{\mathds 1}=\bm{0}$, $\bp \mathds{1}=1$ and hence $\bp\in\Delta^{m-1}$, thus, we set
\[
\bp_k^-(\lambda) = \bp_{k-1}^+(\lambda)(\bm{I}-\lambda\bQ)^{-1}.
\]
Comparing with \eqref{eq:bpkmlambda} we see that
$
\bp_k^-(\lambda) = \bpi^-(k\lambda) + o(\lambda),
$
which is our desired result.
\end{proof}

\begin{remark}
It is possible to use other (weighted) inner products instead of our specific choice (\ref{DefInnerProduct}). For example, \cite[Ch. 6.1.2]{stroock2013introduction} motivates the choice $\langle \bp,\bq\rangle_{\bpi_\infty}:= \sum_{i}(\bp)_{i}(\bq)_{i}(\bpi_\infty)_{i}$, using which in (\ref{Phiminus}) and (\ref{DefNorm}), we again arrive at the statement in Theorem \ref{ThmPrior}. The proof steps are as before, except now the common post-multiplication matrix term in (\ref{laststepPriorProof}) becomes $\bm{D}_{\bpi_\infty}$ instead of $\bm{D}_{\bpi_\infty}^{-1}$.  
\end{remark}

\begin{figure}[htpb]
\centering
\begin{subfigure}{.49\textwidth}
  \centering
  \includegraphics[width=0.84\linewidth]{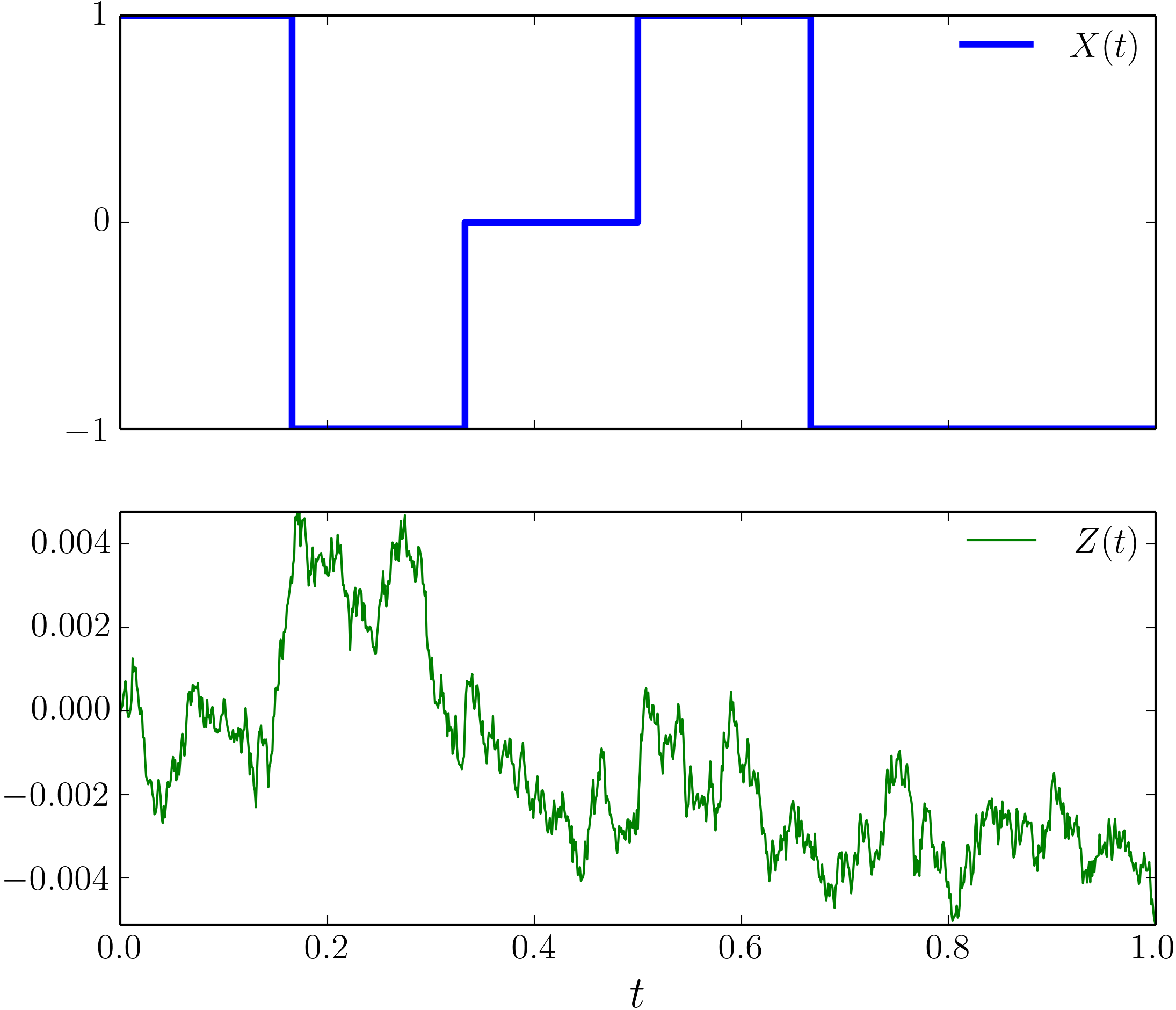}
  \caption{\small{A sample path of the state $X(t)$ (\emph{top}) and of the observation process $Z(t)$ (\emph{bottom}) shown for Example 1 in Section \ref{SecNumericalExample}.}}
  \label{RevXZ}
\end{subfigure}
\par\bigskip
\begin{subfigure}{.49\textwidth}
  \centering
  \includegraphics[width=0.84\linewidth]{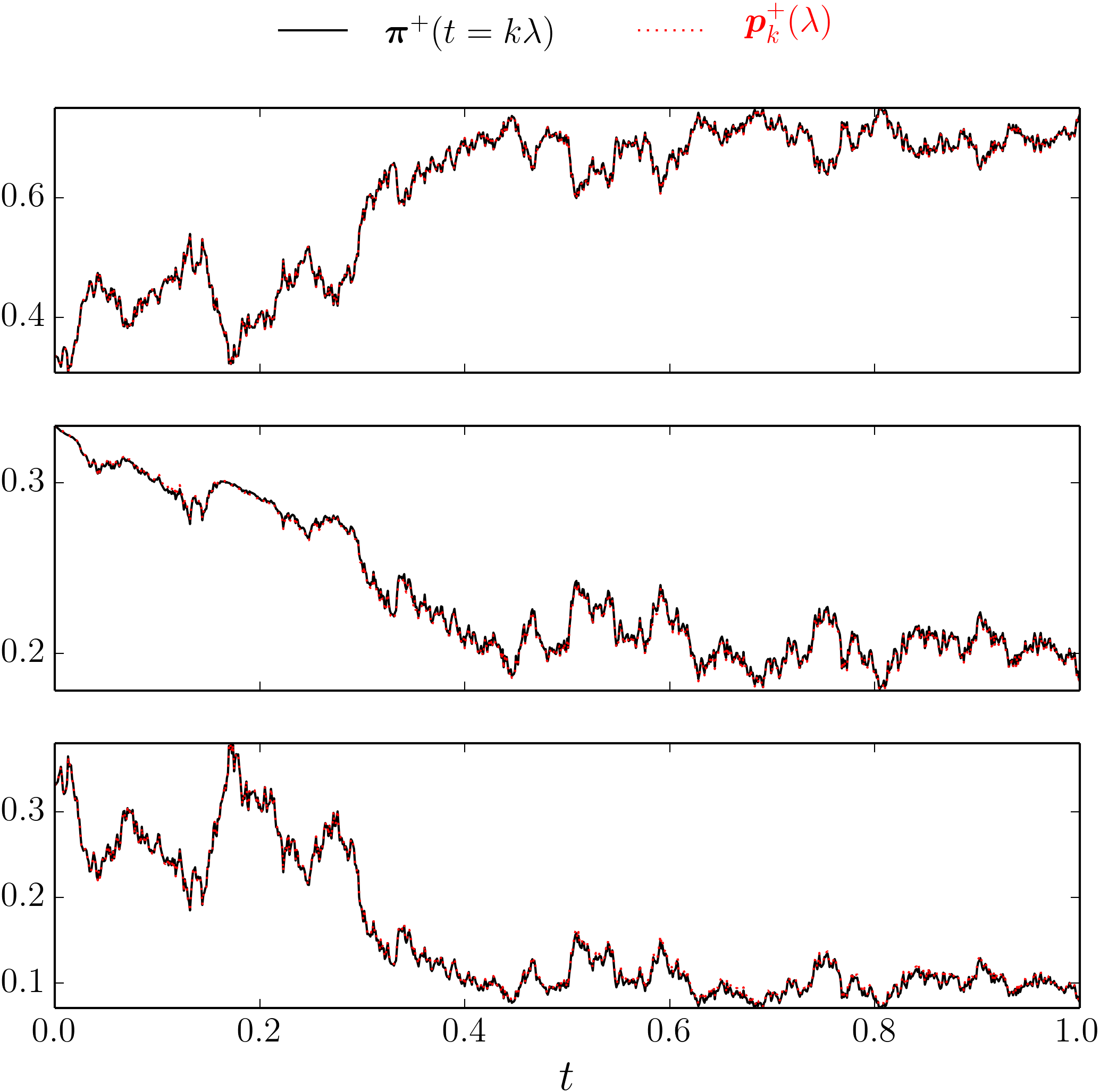}
  \caption{\small{Starting from the initial occupation probability vector $(1/3,1/3,1/3)$, shown above are sample paths for the first (\emph{in top}), the second (\emph{in middle}), and the third (\emph{in bottom}) component of the true (\emph{black, solid}) and approximate (\emph{red, dashed}) posterior probability vectors for Example 1 in Section \ref{SecNumericalExample}.}}
  \label{RevPosterior}
\end{subfigure}
\caption{Simulation results for Example 1 in Section \ref{SecNumericalExample}.}
\label{RevExample}
\vspace*{-0.2in}
\end{figure}  
\begin{figure}[htpb]
\centering
\begin{subfigure}{.49\textwidth}
  \centering
  \includegraphics[width=0.84\linewidth]{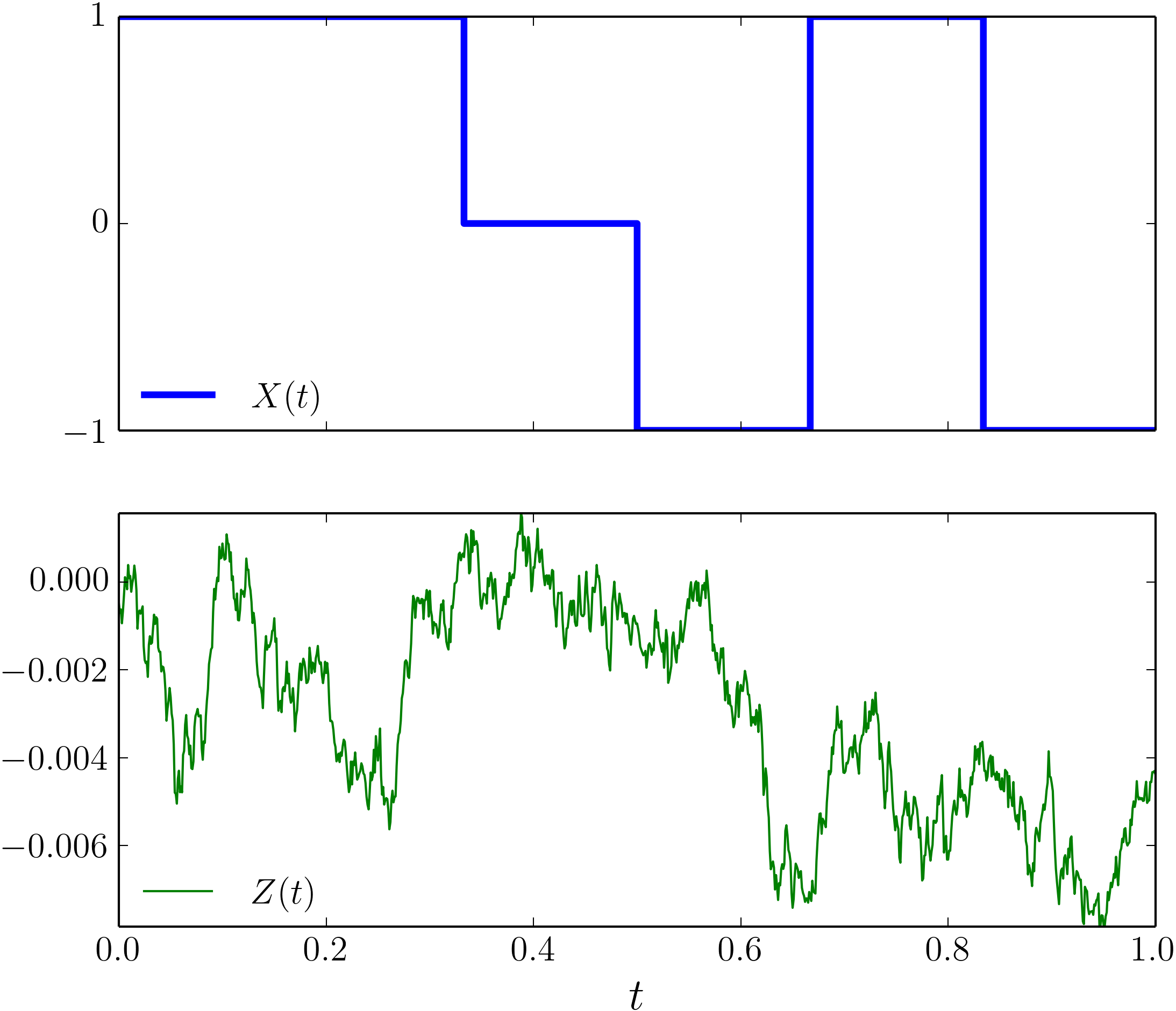}
  \caption{\small{A sample path of the state $X(t)$ (\emph{top}) and of the observation process $Z(t)$ (\emph{bottom}) shown for Example 2 in Section \ref{SecNumericalExample}.}}
  \label{NonrevXZ}
\end{subfigure}
\par\bigskip
\begin{subfigure}{.49\textwidth}
  \centering
  \includegraphics[width=0.84\linewidth]{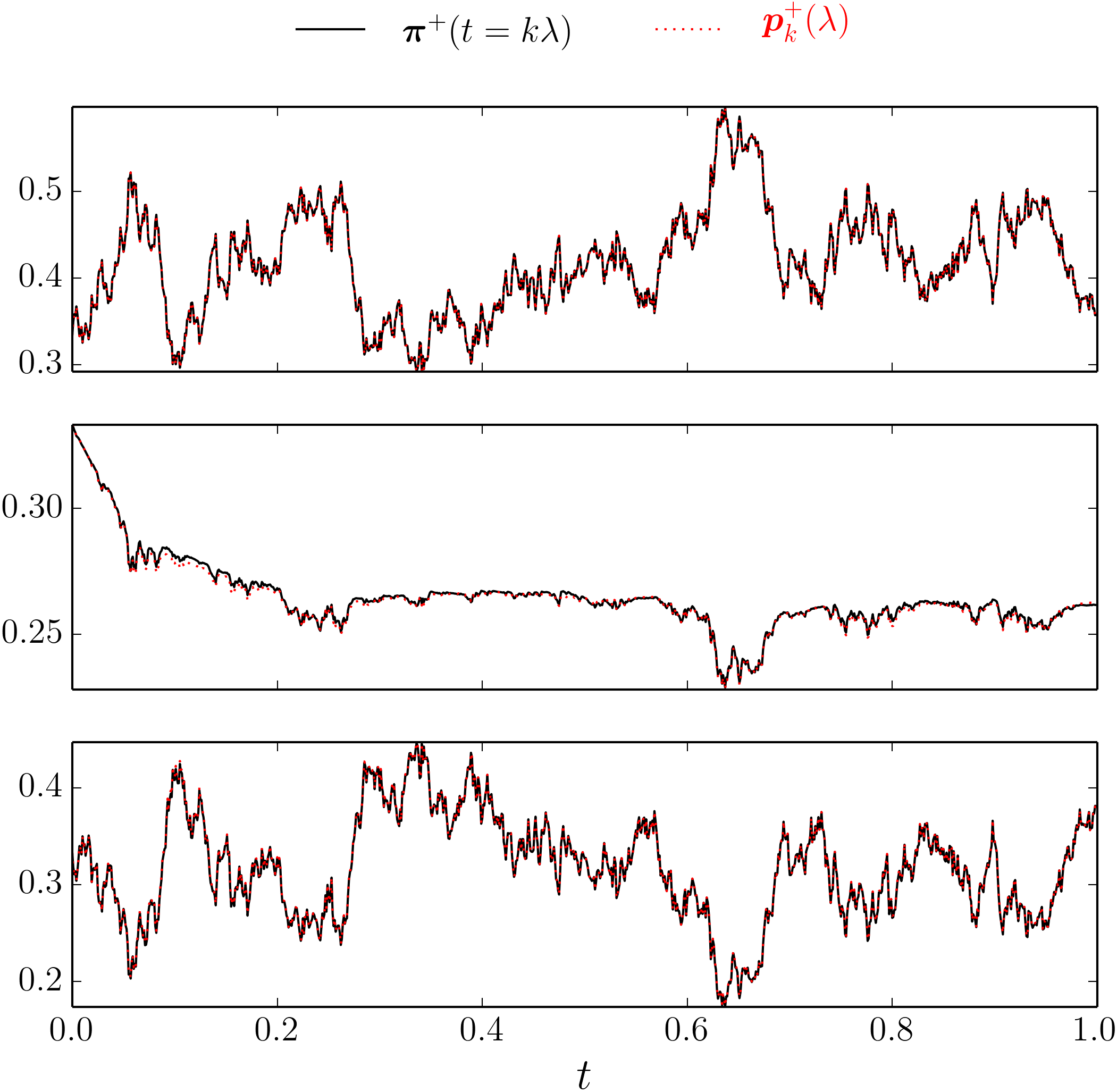}
  \caption{\small{Starting from the initial occupation probability vector $(1/3,1/3,1/3)$, shown above are sample paths for the first (\emph{in top}), the second (\emph{in middle}), and the third (\emph{in bottom}) component of the true (\emph{black, solid}) and approximate (\emph{red, dashed}) posterior probability vectors for Example 2 in Section \ref{SecNumericalExample}.}}
  \label{NonrevPosterior}
\end{subfigure}
\caption{Simulation results for Example 2 in Section \ref{SecNumericalExample}.}
\label{NonrevExample}
\vspace*{-0.2in}
\end{figure}  

\section{Numerical Examples}\label{SecNumericalExample}
The examples below concern estimating the state of a 3-state continuous time Markov chain, the first one being reversible while the second is not.\\[.1in]
\noindent{\em A. Example 1}\\
We consider estimating the state of an irreducible Markov chain $X(t)$ taking values in $\{-1,0,1\}$ with rate matrix
\begin{align}
\bQ = \begin{bmatrix}
-1 & 1/2 & 1/2\\
2 & -2 & 0\\
3 & 0 & -3
 \end{bmatrix}.
\label{RevQ}	
\end{align}
It is easy to verify that the stationary distribution $\bm{\pi}_{\infty}=\left(12/17, 3/17, 2/17\right)$, and that the reversibility condition \eqref{eq:detail} holds. In (\ref{ObsModel}), we set $h(X(t))=0.01 X(t)$, and $\sigma_{V}=0.01$. Fig. \ref{RevXZ} shows a sample paths for $X(t)$ and $Z(t)$. In Fig. \ref{RevPosterior}, we compare the time evolution of the components of the associated posterior probability vector $\bm{\pi}^{+}(t)$ from the Wonham filter (\emph{in black, solid}), with the same of the approximator $\bm{p}^{+}_{k}(\lambda)$ (\emph{in red, dashed}) computed via the proposed proximal recursion framework (Fig. \ref{BlockDiagm}) with step-size $\lambda=10^{-3}$ and $t\in[0,1]$. We only show here the result for a fixed initial condition $\bm{\pi}_{0}=(1/3,1/3,1/3)$; the trends are similar for other initial conditions. Computing $\bm{\pi}^{+}(t)$ in the Wonham filter entails numerically solving the system of coupled nonlinear SDEs (\ref{WonhamPosteriorSDE}), done here via the Euler-Maruyama method. In contrast, computing $\bm{p}^{+}_{k}(\lambda)$ entails recursive evaluation of (\ref{EulerPrior}) and (\ref{KLproxAnalytical}). In our numerical experiments, the latter was observed to enjoy about an order of magnitude computational speed up. Fig. \ref{RevPosterior} shows that the respective posterior sample paths match, as predicted.\\[.1in]
\noindent{\em B. Example 2}\\
Next we consider a Markov chain $X(t)$ taking values in $\{-1,0,1\}$ with rate matrix
\begin{align}
\bQ = \begin{bmatrix}
-5 & 3 & 2\\
4 & -10 & 6\\
3 & 4 & -7 	
 \end{bmatrix},
\label{NonrevQ}	
\end{align}
and $h(\cdot)$, $\sigma_{V}$, $\lambda$ as before. Here again, for $t\in[0,1]$, we compare the posterior sample paths computed from the Wonham filter (\ref{WonhamPosteriorSDE}) with its approximator computed via the proposed proximal recursion approach. Notice that for prior computation, although one does not have a metric version of the variational formula (\ref{PriorProx}), the approximation (\ref{EulerPrior}) still applies for small $\lambda$. Thus, $\bm{p}_{k}^{+}(\lambda)$ can still be computed by recursive evaluation of (\ref{EulerPrior}) and (\ref{KLproxAnalytical}).   

Fig. \ref{NonrevXZ} shows a sample path for $X(t)$, and the same for $Z(t)$. In Fig. \ref{NonrevPosterior}, we show the corresponding $\bm{\pi}^{+}(t)$ from the Wonham filter (\emph{in black, solid}), and $\bm{p}^{+}_{k}(\lambda)$ (\emph{in red, dashed}) from the proximal recursion for this case, starting from the initial condition $\bm{\pi}_{0}=(1/3,1/3,1/3)$. The respective sample paths are in agreement, as expected.


\section{Conclusions}\label{SecConclusion}
This purpose of this paper is to expand on the list of examples where the governing equations for state estimation, conditioned on the history of noisy measurements, can be expressed as gradient flow on the space of probability density functions with respect to a suitable metric. This viewpoint promises a new class of estimation algorithms, taking advantage of implementation of flows via recursive application of proximal projections. Prior work elucidated the case of the Kalman-Bucy filter, and therefore, in the present work we sought to extend the paradigm to case of the Wonham filter. The latter estimates the state of a continuous time Markov chain on a finite state space based on noisy observations. In this paper we have established that the posterior flow that is provided by the Wonham filter can be expressed as the small time-step limit of proximal recursions of certain functionals on the probability simplex. Our preliminary numerical experiments reported here hint at possible computational advantages of the proposed approach, especially for large Markov chains. This will be systematically investigated in our future work.




\section*{Appendix}

\subsection{Proof of Theorem \ref{ThmPosteriorSplProxWonham}}
\begin{proof}[\unskip\nopunct]
With the stated choices for the pair $(d^{+},\Phi^{+})$, we are led to the proximal map of the form
\begin{align}
\underset{\bp\in\Delta^{m-1}}{\arg\inf}\:D_{{\rm{KL}}}\left(\bp\parallel\bp_{k}^{-}\right) + \langle\bc_{k-1},\bp\rangle,\quad k\in\mathbb{N},	
\label{EntropicProx}	
\end{align}
where the $i$-th component of the row vector $\bc_{k-1}$ is
\begin{align}
\bc_{k-1}(i) := \frac{\lambda}{2\left(\sigma_{V}(t_{k-1})\right)^{2}}\left(Y_{k-1} - h(a_{i})\right)^{2}, \, i=1,\hdots,m.
\label{Defc}	
\end{align}
The objective in (\ref{EntropicProx}) is strictly convex since $D_{{\rm{KL}}}$ is strictly convex in $\bp$, and the other summand is linear in $\bp$. Therefore, the $\arg\inf$ in (\ref{EntropicProx}), which we denote by $\bp_{k}^{+}$, is unique. By direct calculation (setting the gradient of Lagrangian w.r.t. $\bp$ to zero, and enforcing the constraint $\bp{\mathds{1}}=1$), we get
\begin{align}
\bp_{k}^{+} = \bp_{k}^{-}\odot\exp\left(-\bc_{k-1}\right)/\left(\left(\bp_{k}^{-}\odot\exp(-\bc_{k-1})\right){\mathds{1}}\right).
\label{KLproxAnalytical}	
\end{align}

Since in this case, we have no prior dynamics, therefore $\bp_{k}^{-} \equiv \bp_{k-1}^{+}$ for all $k\in\mathbb{N}$. Hence, we can rewrite (\ref{KLproxAnalytical}) in terms of the prior probability distribution $\bp_{0}$ as  
\begin{align}
\bp_{k}^{+} = \bp_{0}\odot\exp(-\bgamma_{k-1})/\left(\left(\bp_{0}\odot\exp(-\bgamma_{k-1})\right){\mathds{1}}\right), 
\label{PosteriorGenerealForm}	
\end{align}
where the $1\times m$ vector
\begin{align}
\bgamma_{k-1} := \sum_{r=1}^{k} \bc_{r-1}.
\label{Defgamma}	
\end{align}
Noting that (\ref{PosteriorGenerealForm}) is simply normalization (i.e., Kullback-Leibler projection onto probability simplex) of the vector $\bp_{0}\odot\exp(-\bgamma_{k-1})$, we now unpack $\bgamma_{k-1}$ as function of $\lambda$ and the sampled process $\{Z_{k-1}\}_{k\in\mathbb{N}}$. 

Because $h(\cdot)$ is injective, the random variable $\widetilde{X}:=h(X)$ takes values in $\{\widetilde{a}_{1}, \hdots, \widetilde{a}_{m}\}$, and $\widetilde{X}=\widetilde{a}_{i}:=h(a_{i})$ with probability $\bp_{0}(i)$. Combining (\ref{Defc}) and (\ref{Defgamma}), we then have
\begin{align}
\bgamma_{k-1}(i) = \frac{\lambda}{2}\displaystyle\sum_{r=1}^{k}\frac{\left(Y_{r-1} - \widetilde{a}_{i}\right)^{2}}{\left(\sigma_{V}(t_{r-1})\right)^{2}}, \quad i=1,\hdots,m.
\label{gammacomp}	
\end{align}
Expanding the square in the numerator of each summand in (\ref{gammacomp}), substituting $Y_{r-1} = (Z_{r} - Z_{r-1})/\lambda$ for $r=1,\hdots,k$, and rearranging yields
\begin{align}
\bgamma_{k-1}(i) = &\underbrace{\frac{1}{2}\displaystyle\sum_{r=1}^{k}\frac{\left(Z_{r} - Z_{r-1}\right)^{2}}{\lambda\left(\sigma_{V}(t_{r-1})\right)^{2}}}_{\text{term 1}} - \underbrace{\widetilde{a}_{i}\displaystyle\sum_{r=1}^{k}\frac{Z_{r} - Z_{r-1}}{\left(\sigma_{V}(t_{r-1})\right)^{2}}}_{\text{term 2}}\nonumber\\
&+ \underbrace{\frac{1}{2}\widetilde{a}_{i}^{2}\displaystyle\sum_{r=1}^{k}\frac{\lambda}{\left(\sigma_{V}(t_{r-1})\right)^{2}}}_{\text{term 3}}.
\label{gammacompsimplified}	
\end{align}

To simplify the term 1 indicated in (\ref{gammacompsimplified}), we use the Euler-Maruyama update for (\ref{Observn}) given by
\begin{align}
Z_{r} = Z_{r-1} + h(a_{i})\lambda + \sigma_{V}(t_{r-1})\left(V_{r} - V_{r-1}\right) + O(\lambda^{2}), 
\label{EMupdateforObservn}	
\end{align}
where $r=1,\hdots,k$, and the increments $(V_{r} - V_{r-1})$ are i.i.d. zero mean normal random variables with variance $\lambda$. From (\ref{EMupdateforObservn}), we get
\begin{align}
\left(Z_{r} - Z_{r-1}\right)^{2} = \left(\sigma_{V}(t_{r-1})\right)^{2}\lambda,  
\label{DeltaZsq}	
\end{align}
where we used\footnote{These are equivalent to the well-known relations in stochastic calculus: $(\differential t)^{2}=0$, $(\differential V)^{2} = \differential t$, and $\differential V\differential t = 0$.} $\lambda^{2}=0$, $(V_{r} - V_{r-1})^{2}=\lambda$, and $(V_{r} - V_{r-1})\lambda = 0$. Consequently, term 1 in (\ref{gammacompsimplified}) equals $k/2$.

Combining (\ref{PosteriorGenerealForm}), (\ref{gammacompsimplified}), and (\ref{DeltaZsq}), we thus obtain
\begin{align}
\bp_{k}^{+}(i) = \displaystyle\frac{\bp_{0}(i)\exp(-k/2 + \text{term 2} - \text{term 3})}{\displaystyle\sum_{i=1}^{m}\bp_{0}(i)\exp(-k/2 + \text{term 2} - \text{term 3})}.
\label{PosteriorComponentForm}	
\end{align}
Passing (\ref{PosteriorComponentForm}) to the limit $\lambda\downarrow 0$, recalling that $\widetilde{a}_{i} = h(a_{i})$, and that the time interval $[0,t]$ was divided into sub-intervals with breakpoints $t_{0}\equiv 0, t_{1}, \hdots, t_{k}\equiv t$, we arrive at
\begin{align}
&\displaystyle\lim_{\lambda\downarrow 0} \bp_{k}^{+}(i) = \nonumber\\
&\displaystyle\frac{\bp_{0}(i)\exp\!\left(\!h(a_{i})\int_{0}^{t}\frac{\differential Z(s)}{\left(\sigma_{V}(s)\right)^{2}} - \frac{1}{2}(h(a_{i}))^{2}\int_{0}^{t}\frac{\differential s}{\left(\sigma_{V}(s)\right)^{2}}\!\right)}{\sum_{i=1}^{m}\bp_{0}(i)\exp\!\left(\!h(a_{i})\int_{0}^{t}\frac{\differential Z(s)}{\left(\sigma_{V}(s)\right)^{2}} - \frac{1}{2}(h(a_{i}))^{2}\int_{0}^{t}\frac{\differential s}{\left(\sigma_{V}(s)\right)^{2}}\!\right)}.
\label{PosteriorFinal}	
\end{align}  
The right-hand-side of (\ref{PosteriorFinal}) is exactly the solution of the SDE (\ref{WonhamPosteriorSDE}) for $\bQ\equiv\bm{0}$ with initial condition $\bpi^{+}(0)=\bp_{0}$; see \cite[Appendix 2]{wonham1964some} for a proof. Therefore, we conclude that $\displaystyle\lim_{\lambda\downarrow 0} \bp_{k}^{+} = \bpi^{+}(t=k\lambda)$, as desired.
\end{proof}


\subsection{Proof of Theorem \ref{ThmPosteriorGeneralProxWonham}}
\begin{proof}[\unskip\nopunct]
We start by noting that the development in Appendix A up to expression (\ref{KLproxAnalytical}) still applies. Also, since $h(\cdot)$ is injective, the process $h(X(t))\sim\text{Markov}(\bQ)$ takes values in $\{h(a_{1}),\hdots,h(a_{m})\}$.

For $X(t)\sim\text{Markov}(\bQ)$, the map $\bp_{k-1}^{+}\mapsto\bp_{k}^{-}$, $k\in\mathbb{N}$, corresponding to the Euler discretization of (\ref{PriorEvolution}) is
\begin{align}
\bp_{k}^{-} = \bp_{k-1}^{+}\left(\bm{I} + \lambda\bQ\right) + O(\lambda^{2}).
\label{EulerPrior}	
\end{align}
Let $\Delta Z_{k-1}:=Z_{k} - Z_{k-1}$. Substituting $Y_{k-1}=\Delta Z_{k-1}/\lambda$ in (\ref{Defc}), expanding the square, and using (\ref{DeltaZsq}), we have
\begin{align}
\exp(-\bc_{k-1}(i)) = &\exp(-1/2)\times\exp\left(\!\frac{h(a_{i})\Delta Z_{k-1}}{\left(\sigma_{V}(t_{k-1})\right)^{2}}\!\right)\times\nonumber\\
&\exp\left(-\frac{\lambda (h(a_{i})^{2}}{2\left(\sigma_{V}(t_{k-1})\right)^{2}}\right),
\label{expminusc}	
\end{align}
for $i=1,\hdots,m$. 

Up to first order, the second exponential factor in (\ref{expminusc}) approximates as $1 + h(a_{i})\Delta Z_{k-1}/\left(\sigma_{V}(t_{k-1})\right)^{2}$. For the third exponential factor in (\ref{expminusc}), notice that since $\lambda\equiv\differential t$ for $\lambda\downarrow 0$, therefore from (\ref{ObsModel}), we have $h^{2}(a_{i}) = (\differential Z - \sigma_{V}\differential V)^{2}/\lambda^{2} = 0$, wherein we used $(\differential Z)^{2} = \sigma_{V}^{2}\lambda$, $(\differential V)^{2}=\lambda$, and $\lambda\differential V = 0$. Putting these together, (\ref{expminusc}) yields
\begin{align}
\exp(-\bc_{k-1}(i)) \approx \exp(-1/2) \left(1 + \frac{h(a_{i})\Delta Z_{k-1}}{\left(\sigma_{V}(t_{k-1})\right)^{2}}\right).
\label{approxExpMinusc}	
\end{align}

Substituting for $\bp_{k}^{-}$ and $\exp(-\bc_{k-1})$ in (\ref{KLproxAnalytical}) from (\ref{EulerPrior}) and (\ref{approxExpMinusc}), respectively, we get
\begin{align}
\bp_{k}^{+}(i) = \nu(i)/\delta, \quad i=1,\hdots,m,
\label{Intermedpkplus}	
\end{align}
\vspace*{-0.2in}
\begin{subequations}
\begin{align}
\nu(i) \!&:= \!\!\left(\!\bp_{k-1}^{+}(i) + \lambda\!\displaystyle\sum_{j=1}^{m}\!\bp_{k-1}^{+}(j)\bQ(i,j)\!\right)\!\!\left(\!1 + \frac{h(a_{i})\Delta Z_{k-1}}{\left(\sigma_{V}(t_{k-1})\right)^{2}}\!\right),\label{num}\\
\delta &:= \displaystyle\sum_{i=1}^{m}\nu(i).
\label{deno}
\end{align}	
\label{NumDeno}		
\end{subequations}
From (\ref{EMupdateforObservn}), we observe that $\lambda\Delta Z_{k-1} = h(a_{i})\lambda^{2} + \sigma_{V}(t_{k-1})\lambda(V_{k}-V_{k-1}) = 0$, as both $\lambda^{2}$ and $\lambda(V_{k}-V_{k-1})$ are zero. This allows us to simplify (\ref{num}) as
\begin{align}
\nu(i) = \bp_{k-1}^{+}(i) &+ \frac{\Delta Z_{k-1}}{\left(\sigma_{V}(t_{k-1})\right)^{2}}h(a_{i})\bp_{k-1}^{+}(i) \nonumber\\
&+ \lambda\!\displaystyle\sum_{j=1}^{m}\!\bp_{k-1}^{+}(j)\bQ(i,j) + O(\lambda^2),
\label{NumSimplified}	
\end{align}
and consequently, (\ref{deno}) reduces\footnote{We use here that $\sum_{i=1}^{m}\bQ(i,j)=0$ for all $j=1,\hdots,m$.} to
\begin{align}
\delta = \displaystyle\sum_{i=1}^{m}\nu(i) = 1 + \frac{\Delta Z_{k-1}}{\left(\sigma_{V}(t_{k-1})\right)^{2}}\widehat{h}(t_{k-1}).
\label{DenoSimplified}	
\end{align}
 
Combining (\ref{Intermedpkplus}), (\ref{NumSimplified}) and (\ref{DenoSimplified}), we obtain
{\footnotesize
\begin{align}
\bp_{k}^{+}(i) - \bp_{k-1}^{+}(i) \!=\! \!\left(\!\frac{\Delta Z_{k-1}}{\left(\sigma_{V}(t_{k-1})\right)^{2}}\bp_{k-1}^{+}\left(\!h(a_{i})-\widehat{h}(t_{k-1})\!\right)\right.\nonumber\\
\left. + \lambda\!\displaystyle\sum_{j=1}^{m}\!\bp_{k-1}^{+}(j)\bQ(i,j)\!\right) \!\times
\!\left(\!1 + \frac{\Delta Z_{k-1}}{\left(\sigma_{V}(t_{k-1})\right)^{2}}\widehat{h}(t_{k-1})\!\right)^{\!-1}.
\label{DifferencePosterior}	
\end{align}
}
Up to first order, the second factor in (\ref{DifferencePosterior}) approximates as $1 - \Delta Z_{k-1}\widehat{h}(t_{k-1})/\left(\sigma_{V}(t_{k-1})\right)^{2}$. Using this approximation together with $(\Delta Z_{k-1})^{2} = \lambda(\sigma_{V}(t_{k-1}))^{2}$ (from (\ref{DeltaZsq})), and that $\lambda\Delta Z_{k-1}=0$ (as before), (\ref{DifferencePosterior}) simplifies as
\begin{align}
\bp_{k}^{+}(i) - &\bp_{k-1}^{+}(i) \!= \lambda\displaystyle\sum_{j=1}^{m}\!\bp_{k-1}^{+}(j)\bQ(i,j) + \frac{\bp_{k-1}^{+}(i)}{\left(\sigma_{V}(t_{k-1})\right)^{2}}\times\nonumber\\
&\left(\!h(a_{i})-\widehat{h}(t_{k-1})\!\right)\!\left(\!\Delta Z_{k-1} - \widehat{h}(t_{k-1})\lambda\!\right),
\label{EulerMaruyamaWonham}	
\end{align}
which is exactly the first order (Euler-Maruyama) discretization of the SDE (\ref{WonhamPosteriorSDE}). Specifically, in the limit $\lambda\downarrow 0$, (\ref{EulerMaruyamaWonham}) reduces to (\ref{WonhamPosteriorSDE}). Hence the statement.
\end{proof}


\end{document}